\documentclass[preprint,12pt,a4paper,colorlinks=true,allcolors=blue]{elsarticle}

\usepackage{amssymb}
\usepackage{amsthm}
\usepackage{hyperref}
\usepackage{paralist, mathdots} 
\usepackage{float}
\usepackage{paralist} 
\usepackage{amsmath,tikz,wasysym,multirow,enumitem,lscape,multirow,subfigure}
\usepackage{mathrsfs}

\newtheorem{theorem}{Theorem}[section]
\newtheorem{corollary}[theorem]{Corollary}
\newtheorem{lemma}[theorem]{Lemma}

\newtheorem{proposition}[theorem]{Proposition}
\theoremstyle{definition}
\newtheorem{example}[theorem]{Example}
\newtheorem{remark}[theorem]{Remark}
\newtheorem{definition}[theorem]{Definition}

\newcommand{\R}{\mathbb{R}}
\newcommand{\mc}{\mathcal}
\newcommand{\rk}{\mathrm{rk}}  
\newcommand{\spr}{\mathrm{st}_+}
\newcommand{\cp}{\mathrm{cp}}
\newcommand{\sign}{\mathrm{sign}}
\newcommand{\C}{\mathscr C}
\newcommand{\norm}[1]{\ensuremath{\left\|{#1}\right\|}}
\DeclareMathOperator{\per}{per}
\DeclareMathOperator{\starc}{star}

\begin{document}

\begin{frontmatter}

\title{{\bf Symmetric Nonnegative Trifactorization of Pattern Matrices}}

\author[inst1,inst2]{Damjana Kokol Bukov\v{s}ek}
\affiliation[inst1]{organization={University of Ljubljana, School of Economics and Business},
          country={Slovenia}}
\affiliation[inst2]{organization={Institute of Mathematics Physics and Mechanics},
            city={Ljubljana},
           country={Slovenia}}

\author[inst3]{Helena \v{S}migoc}
\affiliation[inst3]{organization={School of Mathematics and Statistics, University College Dublin},
          country={Ireland}}

\begin{abstract}
A factorization of an $n \times n$ nonnegative symmetric matrix $A$ of the form $BCB^T$, where $C$ is a $k \times k$ symmetric matrix, and both $B$ and $C$ are required to be nonnegative, is called the Symmetric Nonnegative Matrix Trifactorization (SN-Trifactorization). The SNT-rank of $A$ is the minimal $k$ for which such factorization exists. The SNT-rank of a simple graph $G$ that allows loops is defined to be the minimal possible SNT-rank of all symmetric nonnegative matrices whose zero-nonzero pattern is prescribed by a given graph.

We define set-join covers of graphs, and show that finding the SNT-rank of $G$ is equivalent to finding the minimal order of a set-join cover of $G$. Using this insight we develop basic properties of the SNT-rank for graphs and compute it for trees and cycles without loops. We show the equivalence between the SNT-rank for complete graphs and the Katona problem, and discuss uniqueness of patterns of matrices in the factorization. 
\end{abstract}


\end{frontmatter}

\section{Introduction and Notation}\label{sec:1}

Factorizations of matrices where the factors are required to be entry-wise nonnegative provide a powerful tool in analysing nonnegative data. In this paper we consider a factorization of nonnegative  symmetric matrices, which takes into account  symmetry, nonnegativity and low rank of a matrix. 

Denote by $\R_+$ the set of nonnegative real numbers, by $\R^{n \times m}$ the set of $n \times m$ real matrices, and by $\R_+^{n \times m}$ the set of  $n \times m$ entry-wise nonnegative matrices. Furthermore, we denote
$$\mc S_n^+ = \{A \in \R_+^{n \times n}; A = A^T \}.$$

\begin{definition}
A factorization of $A \in  \mc S_n^+$ of the form $BCB^T$, where $B \in \R_+^{n \times k}$ and $C \in \mc S_k^+$, is  called \emph{Symmetric Nonnegative Trifactorization} of $A$ (\emph{SN-Trifactorization} for short). Minimal possible $k$ in such factorization is 
called the \emph{SNT-rank} of $A$, and is denoted by $\spr(A)$.
\end{definition}

The SN-Trifactorization was studied in \cite{MR4547308}, and is closely related to two other well known factorizations that feature nonnegative factors: Nonnegative Matrix Factorization and Completely Positive Factorization. We refer the reader to \cite{NMFGillis} for the background on the Nonnegative Matrix Factorization and to \cite{MR1986666} for the background on the Completely Positive Factorization.

The zero-nonzero pattern of a nonnegative matrix $A \in  \mc S_n^+$ poses constrains on the zero-nonzero pattern of nonnegative matrices $B \in \R_+^{n \times k}$ and $C \in \mc S_k^+$ satisfying $A=BCB^T$. The aim of this paper is to better understand those restrictions.
 
In the context of matrix patterns, graphs arise naturally. In this work, a graph $G=(V(G),E(G))$ will always be a simple undirected graph that allows loops. Hence, $E(G) \subseteq \{\{i,j\}; i,j \in V(G) \}$, where $\{i\} \in E(G)$ corresponds to a loop on the vertex $i$ in $G$. A vertex $i\in V(G)$ is \emph{isolated}, if $\{i,j\} \not\in E(G)$ for all $j \in V(G)$. In particular, $i$ does not have a loop. The vertex set $V(G)$ will often be $[n]:=\{1,2,\ldots, n\}$.

Let $G$ and $H$ be two simple graphs that allow loops. We denote by $G \cup H$ their disjoint union, by $tG$ the union of $t$ copies of $G$ and by $G \vee H$ the join of $G$ and $H$, i.e. the graph with  $V(G \vee H) = V(G) \cup V(H)$ and $E(G \vee H) = E(G) \cup E(H) \cup \{\{i,j\} ~|~ i \in V(G), j \in V(H)\}$. We use standard notation for graphs without any loops: we denote by $K_n$ the complete graph on $n$ vertices, by $K_{m,n}$ the bipartite graph on $m+n$ vertices, by $P_n$ the path, and by $C_n$ the cycle on $n$ vertices. We denote by $K^\ell_n$ the complete graph on $n$ vertices with all possible loops. 

We denote matrices by capital letters, $A \in \R^{n \times m}$, vectors by bold letters, ${\bf a} \in \R^n$, the zero matrix in $\R^{n \times m}$ by $0_{n \times m}$ and the zero vector in $\R^n$ by ${\bf 0}_n$. 

Let $A, B \in \R^{n \times m}$. Then $A > B$ means 
that $A-B$ is entry-wise positive, and $A \ge B$ denotes entry-wise nonnegativity of $A-B$.  
\emph{The support} of a vector ${\bf a} \in \R^n$ is the set of all indices in $i \in [n]$ for which $a_i \ne 0$. For $A\in \R^{n \times m}$ and $\mc S \subset [n], \mc T \subset [m]$, we denote by $A[\mc S, \mc T]$ the submatrix of $A$ containing entries $a_{ij}$ for all $i \in \mc S, j \in \mc T$. 

After a brief introduction and development of notation in Section \ref{sec:1}, we dedicate Section \ref{sec:pattern} to the definition of the SNT-rank and set-join covers  of simple graphs that allow loops. We prove that finding the SNT-rank of a graph $G$ is equivalent to determining minimal order of the set-join cover of $G$. Using this insight we develop basic properties of SNT-rank of graphs. 
We conclude Section \ref{sec:pattern} by defining uniqueness of optimal set-join covers of graphs. In Section \ref{sec:without-loops} we find the SNT-rank of trees  without loops and consider the unicyclic graphs  without loops. Section \ref{sec:complete} is dedicated to the SNT-rank and optimal set-join covers of complete graphs without loops. Section \ref{sec:motivation} offers a selection of possible applications and further research directions motivated by the work in this paper. 

\section{Set-join covers and $\spr(G)$}\label{sec:pattern}

 In this work we will for the most part take aside the actual values of matrices and only consider their patterns. Our main focus will be the question, how small can $\spr(A)$ be among all matrices $A$ with a given zero-nonzero pattern. 
 
\subsection{Setup}

We begin by defining $\spr(G)$ for a  graph $G$, and developing a combinatorial question that is equivalent to finding $\spr(G)$. 
 
\begin{definition}
Let $A \in \R_+^{n \times m}$. The \emph{pattern matrix} of $A$ is the matrix $\mathrm{sign}(A) \in \{0,1\}^{n \times m}$ defined by
$$ \sign(A)_{ij} =  \begin{cases} 1; &\mbox{if } a_{ij} > 0, \\
0; & \mbox{if } a_{ij} = 0. \end{cases}$$
(The pattern matrix is in some literature called \emph{the derangement matrix}.)
\end{definition}

\begin{definition}
\emph{The pattern graph} $G(A)=(V(G), E(G))$ of a matrix  $A \in {\mc S}_n^+$ is defined by $V(G)=[n]$, and $\{i,j\} \in E(G)$ precisely when $a_{ij} > 0$. 
\end{definition}

Let $G$ be a graph with $|V(G)|=n$. By $\mc S^+(G)$ we  denote the set of all matrices in $\mc S^+_n$ with the pattern graph $G$, and ask the question how small can the SNT-rank be on this set. 

\begin{definition}
Let $G=(V(G),E(G))$ be a simple graph with loops. We define:
$\spr(G):=\min\{\spr(A); A \in \mc S^+(G)\}$.
\end{definition}

\begin{remark}
A factorization of $A \in  \R_+^{n \times m}$ of the form $UV^T$, where $U \in \R_+^{n\times k}$ and $V \in \R_+^{m \times k}$, is  called Nonnegative Matrix Factorization of~$A$. Minimal possible $k$ in such factorization is 
called the NMF-rank of~$A$, and is denoted by $\rk_+(A)$ (see \cite{NMFGillis}). In this context also Boolean rank is studied. It is defined as 
$$ \rk_{01}(A) = \min\{\rk_+(B); B \in  \R_+^{n \times m}, \sign(B) = \sign(A)\}.$$
It is easy to see that Boolean rank is equal to rectangle covering bound, i.e. the minimum number of rectangles needed to cover all nonzero entries in $A$ (see \cite{NMFGillis}).
\end{remark}

We proceed to develop basic properties of $\spr(G)$. We start with a simple proposition which follows from the property that for $A\in \mc S_n^+,B \in \mc S_m^+$ we have $\spr(A \oplus B) =\spr(A)+\spr(B)$, see \cite{MR4547308}.

\begin{proposition} \label{prop:union}
For any graphs $G, H$ we have $\spr(G \cup H) =\spr(G)+\spr(H)$. In particular, $\spr(G \cup K_1)=\spr(G)$ as $\spr(K_1)=0$.
\end{proposition}

The next result connects the patterns of nonnegative matrices $A$, $B$ and $C$, provided $A=BCB^T$. 

\begin{proposition}\label{pattern}
Let $A=BCB^T$ be SN-Trifactorization of $A=(a_{ij}) \in\mathcal S_n^+$. For $i\in [n]$ let $r_i(B)^T \in \R^{1\times k}$ be the $i$-th row of $B$, and $\mc R_i \subseteq \{1,\ldots,k\}$ the support of $r_i(B)$, i.e. $\mc R_i:=\{j; (r_i(B))_j\neq 0\}$. Then $C[\mc R_i, \mc R_j]=0$ if and only if $a_{ij}=0$. 
\end{proposition}

\begin{proof}
Since $a_{ij}=r_i(B)^TCr_j(B)=(r_i(B)[\mc R_i])^TC[\mc R_i,\mc R_j](r_j(B)[\mc R_j])$, and $r_s(B)[\mc R_s]>0$  by the definition of $\mc R_s$, the conclusion follows. 
\end{proof}

Note that $\sign(B)$  in Proposition \ref{pattern} is the incidence matrix of ${\bf R}=(\mc R_1, \ldots, \mc R_n)$. We recall the definition of the incidence matrix below. 

\begin{definition}
Let ${\bf  R}=(\mc R_1, \ldots,\mc R_n)$  be a list of sets $\mc R_i \subseteq \{1,\ldots, k\}$. \emph{The incidence matrix} of ${\bf R}$ is an $n \times k$ matrix $\iota({\bf R})$ with
$$\iota({\bf R})_{ij}=\begin{cases}
 1; & j \in \mc R_i, \\
 0; & j \not\in \mc R_i.
\end{cases}$$  
\end{definition}

Since $\spr(G)$ is a parameter of the graph $G$, we would like to determine it directly from $G$. For this we need to introduce the set-join and the set-join cover.

\begin{definition}
Let $\mc S$ be a finite set and $\mc K, \mc L $ two nonempty subsets of $\mc S$ with possibly nonempty intersection. We define \emph{the set-join of $\mc K$ and $\mc L$ on $\mc S$}, denoted by $\mc K \vee_{\mc S} \mc L$, to be the graph with $V(\mc K \vee_{\mc S} \mc L)=\mc S$, and $E(\mc K \vee_{\mc S} \mc L)=\{\{i,j\}; i \in \mc K, j \in \mc L\}$. 
\end{definition}

Note that $\mc K \vee_{\mc S} \mc L$ has a loop on vertex $i$ precisely when $i \in \mc K \cap \mc L$, and that if $\mc K \cup \mc L \neq \mc S$, then $\mc K \vee_{\mc S} \mc L$ contains some isolated vertices without loops.

\begin{definition}
Let $G$ be a graph and $\mc K_i, \mc L_i \subseteq V(G)$. We say that $$\C=\{\mc K_i \vee_{V(G)} \mc L_i, i \in [t]\}$$ is \emph{a set-join cover} of $G$ if  
$E(G)=\bigcup_{i=1}^t E(\mc K_i \vee_{V(G)} \mc L_i)$.

For a set-join cover $\C$ we define:
\begin{itemize}
\item \emph{ the component set} of $\C$ to be: $$V(\C)=\{\mc K_i; i \in [t]\}\cup \{\mc L_i; i \in [t]\},$$
\item \emph{the order} of $\C$ to be $|\C|=|V(\C)|$, 
\item \emph{the graph} of $\C$ denoted by $G(\C)$ to be a graph with $V(G(\C))=V(\C)$ and $\{\mc K_i,\mc K_j\} \in E(G(\C))$ if and only if $\mc K_i \vee_{V(G)} \mc K_j \in \C$.
\end{itemize}

\end{definition}

Note that in the above definition we allow $\mc K_i =\mc K_j$ or $\mc K_i=\mc L_j$, so $|\C|$ can be smaller than $2t$, and can in the extreme case even be equal to $1$. 

\begin{definition}
    Let $G$ be graph with a set-join cover $\C$, and $\mc S \subseteq V(G)$. \emph{The restriction of $\C$ to $\mc S$} is the set
    $$\C[\mc S]=\{(\mc K\cap \mc S)\vee_{V(G)\cap \mc S} (\mc L\cap \mc S); \mc K\vee_{V(G)} \mc L \in \C, \mc K\cap \mc S\not= \emptyset, \mc L\cap \mc S\not= \emptyset\}.$$
\end{definition}

In the definition above, $\C[\mc S]$ is a set-join cover of a subgraph of $G$ induced on $V(G)\cap \mc S$ that satisfies $|\C[\mc S]|\leq |\C|$ . 

\begin{theorem} \label{th:main}
Let $G$ be a graph. Then 
$\spr(G)=\min\{|\C|;\, \C$ a set-join cover of   $G\}$.
\end{theorem}

\begin{proof}
 Let $G$ be a graph with $V(G)=[n]$, and $A \in \mathcal{S}^+(G)$ with SN-Trifactorization $A=BCB^T$, where $C \in \mathcal{S}_k^+$. For $i \in [k]$, let $c_i(B)$ be the $i$-th column of $B$, and $\mathcal L_i$ the support of the $c_i(B)$. Clearly, $c_i(B)c_j(B)^T+c_j(B)c_i(B)^T \in \mc{S}^+(\mc L_i \vee_{[n]} \mc L_j)$, and 
$$\sign(A)=\sign\left\{\sum_{\{i,j\} \in E(G(C))}\big(c_i(B)c_j(B)^T+c_j(B)c_i(B)^T\big)\right\}.$$
We deduce that $\C=\{\mc L_i \vee_{[n]} \mc L_j; \{i,j\}\in E(G(C))\}$ is a set-join cover of $G$ with $|\C|=k$. 

Conversely, let $\C=\{\mc L_i \vee_{[n]} \mc K_i; i \in [t]\}$ be a set-join cover of $G$, and let us list all the elements in $V(\C)$ in some fixed order: $(\widehat{ \mc L}_1,\widehat{ \mc L}_2,\ldots,\widehat{ \mc L}_{|\C|})$. We define $C$ to be a $|\C|\times |\C|$ matrix with:
$$c_{ij}=\begin{cases}
1; &\text{if } \widehat{ \mc L}_i\vee_{[n]}\widehat{ \mc L}_j \in \C, \\
0; &\text{otherwise},
\end{cases}$$
 and $B$ to be an $n \times |\C|$ matrix with 
 $$b_{ij}=\begin{cases}
1, &\text{if } i \in \widehat{\mc L}_j \\
0, &\text{otherwise}.
\end{cases}$$
In other words, $C$ is the zero-one matrix in $\mc S^+(G(\C))$, and $B^T$ is the incidence matrix of $(\widehat{ \mc L}_1,\widehat{ \mc L}_2,\ldots,\widehat{ \mc L}_{|\C|})$.
The proof is completed by noting that $BCB^T \in \mc S^+(G)$. 
\end{proof}

\begin{definition}
    If $\C$ is a set-join cover of $G$ with of order $\spr(G)$, then we say that $\C$ is an \emph{optimal set-join cover} for $G$. We will use the abbreviation \emph{OSJ cover} for $G$. 
\end{definition}

\begin{proposition}\label{prop:OSJgraph}
  An OSJ cover $\C$ of $G$ satisfies $$\spr(G(\C))=|\C|=|V(G(\C))|.$$
\end{proposition}

\begin{proof}
Let $\C$ be an OSJ cover of $G$, $C \in \mc S^+(G(\C))$ and $B = \iota(V(\C))^T$  the transposed incidence matrix of $V(\C)$. Then  $A = BCB^T \in \mc S^+(G)$. If $\spr(G(\C))<|\C|$, then the matrix $C \in \mc S^+(G(\C))$ can be chosen so that $\spr(C) < |\C|$. Hence, $C=B_1C_1B_1^T$ where $C_1 \in \mc S_k^+$ with $k< |\C|$. From $A = (BB_1)C_1(BB_1)^T$ we conclude $\spr(A) \le k < |\C| = \spr(G)$, a contradiction.
\end{proof}

\begin{remark}\label{rem:interpretation}
A set-join cover $\C$ of a graph can be interpreted in the following way. Consider a set of items $V$ that are either required or forbidden to interact. The interactions are organised by meetings of certain subgroups of $V$. If two subgroups $V_1,V_2 \subseteq V$ meet, then all the items from $V_1$ interact with all the items from $V_2$. Hence, if $i,j \in V$ are forbidden to interact, and $V_1$ and $V_2$ are two subgroups that meet, then $i \in V_1$ and $j \in V_2$ is not allowed. The desired interactions can clearly be organised by meetings of singletons, and we are asking what is the minimal number of subgroups that need to be formed, to be able to organise the desired interactions in such a way that no forbidden interactions occur. 

Let $G$ be a graph that records which interactions are required and which are forbidden: $V(G)=V$, $\{i,j\} \in E(G)$ if and only if $\{i,j\}$ are required to interact. Any set-join cover $\C$ of $G$ gives us possible way of organising required interactions, and $\spr(G)$ is the minimal number of groups that need to be formed. 
\end{remark}

We illustrate Theorem \ref{th:main} by the following example. 

\begin{example}\label{ex:K6}
Let $A=BCB^T$ with
$$B=\left(
\begin{array}{ccccc}
 1 & 0 & 1 & 0 & 0 \\
 1 & 0 & 0 & 1 & 0 \\
 1 & 0 & 0 & 0 & 1 \\
 0 & 1 & 1 & 0 & 0 \\
 0 & 1 & 0 & 1 & 0 \\
 0 & 1 & 0 & 0 & 1 \\
\end{array}
\right), \ \ \ C=\left(
\begin{array}{ccccc}
 0 & 1 & 0 & 0 & 0 \\
 1 & 0 & 0 & 0 & 0 \\
 0 & 0 & 0 & 1 & 1 \\
 0 & 0 & 1 & 0 & 1 \\
 0 & 0 & 1 & 1 & 0 \\
\end{array}
\right).$$
Then $G(A) = K_6$. This proves that $\spr(K_6) \le 5$. (In Section \ref{sec:complete} we will show that $\spr(K_6) = 5$.) The patterns of $B$ and $C$ determine the set-join cover of $K_6$:
$$ \mc K_1 = \{1, 2, 3\}, \mc K_2 = \{4, 5, 6\}, \mc K_3 = \{1, 4\}, \mc K_4 = \{2, 5\}, \mc K_5 = \{3, 6\},$$
and $\C=\{\mc K_1 \vee_{[6]} \mc K_2, \mc K_3 \vee_{[6]} \mc K_4, \mc K_4 \vee_{[6]} \mc K_5,\mc K_5 \vee_{[6]} \mc K_3\}$. Notice that $G(\C) = G(C) = K_2 \cup K_3$.
\end{example}

\begin{example} \label{ex:4.2}
Let $G$ be a graph with the adjacency matrix
$$ A= \left(
\begin{array}{ccccc}
 0 & 1 & 1 & 0 & 0 \\
 1 & 1 & 1 & 1 & 0 \\
 1 & 1 & 1 & 1 & 1 \\
 0 & 1 & 1 & 1 & 1 \\
 0 & 0 & 1 & 1 & 0 \\
\end{array}
\right) .$$
Then $\spr(G) = 4$. Indeed, any matrix $X$ with pattern matrix $\sign(X) = A$ has rank at least $4$. The first three rows of $X$ clearly have to be linearly independent, and the last row cannot be written as a linear combination of the first three. Since $\spr(X) \ge \rk(X)$ for any $X \in \mc S^+(G)$, we have $\spr(G) \ge 4$. Let 
$$ \mc K_1 = \{1, 2\}, \mc K_2 = \{2, 3\}, \mc K_3 = \{3, 4\}, \mc K_4 = \{4, 5\},$$
and observe that $\C=\{\mc K_1 \vee_{[5]} \mc K_2, \mc K_2 \vee_{[5]} \mc K_3,\mc K_3 \vee_{[5]} \mc K_4\}$ a set-join cover of $G$ with $|\C|=4$, proving $\spr(G) \le 4$.

Let $B^T = \iota({\bf R})$ be the incidence matrix of the list ${\bf R} = (\mc K_1, \mc K_2, \mc K_3, \mc K_4)$, $C$ the adjacency matrix of the path on 4 vertices, and $X=BCB^T$. Then $X \in \mc S^+(G)$ and 
$\spr(X) = 4$.
\end{example}

\begin{remark}
A matrix $A \in  \mc S_n^+$ is completely positive, if it can be written as $BB^T$ for some matrix $B \in \R_+^{n\times k}$. Such factorization is called CP-Factorization of $A$. The minimal possible $k$ in CP-Factorization as above is called the CP-rank of $A$, and is denoted by $\cp(A)$, see \cite{MR1986666}. 
The set $\mc S_n^+(G)$ contains matrices with CP-rank $k$ if and only if there exists a set-join cover of $G$ of the form $\C=\{\mc K_i \vee_{V(G)} \mc K_i, i=1\ldots,k\}$, since any CP-factorization $A=BB^T$ can be viewed as an SN-Trifactorization $A=BCB^T$, with the middle matrix $C$ equal to the identity.  So the graph $G(\C)$ is equal to $tK_1^\ell$.
Hence, all the edges of $G$ can be covered with $k$ complete graphs with all the loops. The lowest CP-rank that a nonnegative symmetric matrix with prescribed zero-nonzero pattern can have is equal to the clique cover number of the pattern graph.
\end{remark}

\begin{lemma}\label{lem:not-opt}
Let $\C$ be a set-join cover for a graph $G$ with $V(G)=[n]$. Assume there exist:
\begin{itemize}
\item $V' \subseteq V(\C)$ with $|V'|=t$, and
\item $V=\{\mc L_i \subseteq [n];\,   i \in [s]\}$,
\end{itemize}
so that $s<t$ and every element of $V'$ can be written as the union of some elements of $V$. Then $\C$ is not an OSJ cover of $G$.
\end{lemma}

\begin{proof}
We will prove the lemma by constructing a set-join cover $\widehat\C$ of $G$ with $V(\widehat\C)=(V(\C)\setminus V')\cup V$. This will prove our claim as $|(V(\C)\setminus V')\cup V|\le|\C|-t+s<|\C|$. We define $\widehat\C$ as the union of the following three sets:
\begin{itemize}
\item $\C_1:=\{\mc K \vee_{[n]} \mc K';$  $\mc K \vee_{[n]} \mc K' \in \C$ and $\mc K, \mc K' \in (V(\C)\setminus V')\}$. 
\item  $\C_2:=\{\mc K \vee_{[n]} \mc L;$ $\mc K \in (V(\C)\setminus V')$, $\mc L \in V$ and there exists $\mc K \vee_{[n]} \mc K' \in \C$ with $\mc L \subseteq \mc K'\}$, 
\item  $\C_3:=\{\mc L \vee_{[n]} \mc L';$ $\mc L, \mc L' \in V$ and there exists $\mc K \vee_{[n]} \mc K' \in \C$ with $\mc L \subseteq \mc K$, and $\mc L' \subseteq \mc K'\}$. 
\end{itemize}
To prove that $\widehat \C$ is a set-join cover of $G$ we need to show that $\widehat\C$ covers all the edges of $G$, and that it does not cover any edges that are not in $G$. Since $\C$ is a cover of $G$, it is clear that $\C_1$ does not cover any edges that are not in $E(G)$. If $\mc K \in (V(\C)\setminus V')$, $\mc L \in V$ and there exists $\mc K \vee_{[n]} \mc K' \in \C$ with $\mc L \subseteq \mc K'$, then $E(\mc K \vee_{[n]} \mc L) \subseteq E(\mc K \vee_{[n]} \mc K')\subseteq E(G).$ Hence, $\C_2$ does not cover any edges not in $E(G)$. The claim for $\C_3$ is proved in a similar way. 

Now let $e \in E(G)$ and $\mc K \vee_{[n]} \mc K' \in \C$ with $e \in E(\mc K \vee_{[n]} \mc K')$. If $\mc K, \mc K' \in V(\C)\setminus V'$, then $\mc K \vee_{[n]} \mc K' \in \C_1$. If $\mc K \in V(\C)\setminus V'$ and $\mc K' \in V'$, then $\mc K'$ is the union of some elements form $V$, hence there exists $\mc L \in V$ so that $\mc L \subseteq \mc K'$ and $e \in \mc K \vee_{[n]} \mc L$, hence $e$ is covered by $\C_2$. The case when  $\mc K, \mc K' \in V'$ is proved in a similar way. 
\end{proof}

Let $\C$ be a set-join cover of a graph $G$ with the component set $V(\C)=\{\mc K_i, i \in [t]\}$. Recall that \emph{a system of distinct representatives} (SDR for short) of $V(\C)$ is a set $\{x_i, i \in [t]\}$ with the property that $x_i \in \mc K_i$  and $x_i$ are distinct, \cite[Section 1.2]{MR3469704}. 

\begin{lemma}\label{lem:SDR}
Let $\C$ be an optimal set-join cover of a graph $G$. Then $V(\C)$ has a system of distinct representatives.
\end{lemma}
\begin{proof}
It is well known that a family of subsets $V(\C)$ has an SDR if an only the permanent of the incidence matrix $\iota(V(\C))$ is not zero, \cite[Section 7.5]{MR3469704}.
Assume then that $\C$ is an optimal cover of $G$ with $\per(\iota(V(\C)))=0$. Since $\spr(G) \leq n$, this is equivalent to  $\iota(V(\C))$ having an $w \times t$ zero submatrix, where $w+t=n+1$, by Frobenius-K\"{o}nig Theorem, see for example \cite[Section 1.2]{MR3469704}. Hence, there exists $V'\subset V(\C)$ with $|V'|=t$ so that $|\cup_{\mc K_i \in V'} \mc K_i| \leq n-w=t-1$. Let $V = \{\{x\}, x\in \cup_{\mc K_i \in \mc V} \mc K_i\}$ be the set of singletons from $\cup_{\mc K_i \in V'} \mc K_i$. Then $V'$ and $V$ satisfy the conditions of Lemma \ref{lem:not-opt}, thus $\C$ is not optimal, a contradiction.
\end{proof}

\begin{theorem}\label{thm:subgraph}
Let $\C$ be an optimal set-join cover of a graph $G$. Then $G$ contains a subgraph that is isomorphic to $G(\C)$. 
\end{theorem}

\begin{proof}
Let $\C$ be an optimal cover for $G$ and $\mc S=\{x_i, i=1\ldots,|\C|\}$ an SDR for $V(\C)$  that exists by Lemma \ref{lem:SDR}. Clearly, there exists a (not necessarily induced) subgraph of $G$ on vertices from $\mc S$ that is isomorphic to $G(\C)$. 
\end{proof}

\begin{remark}
A graph property $P$ is \emph{monotone} if every subgraph of a graph with property $P$ also has property $P$.
Theorem \ref{thm:subgraph} shows that if $G$ has a monotone graph property and $\C$ is an OSJ cover of $G$, then $G(\C)$ has it also. For example:
\begin{itemize}
\item If $G$ is a forest, then $G(\C)$ is a forest. 
\item If $G$ is triangle free, then $G(\C)$ is also. 
\item If $G$ is bipartite, then $G(\C)$ is also. 
\end{itemize}
\end{remark}

\subsection{Uniqueness}

After we establish $\spr(G)$,  we can ask, if the set-join cover of $G$ of order $\spr(G)$ is unique. We will consider three different types of uniqueness as described in the definition below.  

\begin{definition}
A graph $G$ has \emph{unique optimal set-join cover} (\emph{unique OSJ cover} for short), if the OSJ cover of $G$ of order $\spr(G)$ is unique. 
A graph $G$ has \emph{essentially unique OSJ cover}, if for any two covers $\C$ and $\C'$ of $G$ satisfying $|\C|=|\C'|=\spr(G)$, there exists an automorphism $\sigma: V(G) \rightarrow V(G)$ of $G$ so that $\sigma(\C)=\C'$. 
(For a cover $\C=\{\mc K_i \vee_{V(G)} \mc L_i, i=1\ldots,\spr(G)\}$, we define $\sigma(\C)$ to be the cover $\{\sigma(\mc K_i)~\vee_{V(G)}~ \mc\sigma(\mc L_i), i=1\ldots,\spr(G)\}$.)

A graph $G$ has the \emph{unique OSJ cover graph}, if all covers $\C$ of $G$ of order $\spr(G)$ have the same $G(\C)$ up to isomorphism of graphs. 
\end{definition}

\begin{example}
Let $G_1 = (3K_1^{\ell}) \vee K_1^{\ell}$ be the star graph on $4$ vertices with all the loops, and denote the vertices of $3K_1^{\ell}$ by $1$, $2$, $3$ and the central vertex by $4$. Then $\spr(G_1)=3$ and $G_1$ has the unique OSJ cover $\C_1$, with the components: 
$$ \mc K_1 = \{1, 4\}, \mc K_2 = \{2, 4\}, \mc K_3 = \{3, 4\},$$
$\C_1=\{\mc K_1 \vee_{[4]} \mc K_1, \mc K_2 \vee_{[4]} \mc K_2,\mc K_3 \vee_{[4]} \mc K_3\}$ and $G(\C_1)=3K_1^{\ell}$.

Next, let $G_2$ be defined by $V(G_2) = [4]$ and adjacency matrix 
$$ A_2= \left(
\begin{array}{ccccc}
 1 & 1 & 1 & 1  \\
 1 & 1 & 0 & 1  \\
 1 & 0 & 0 & 1 \\
 1 & 1 & 1 & 0 \\
\end{array}
\right).$$
Again $\spr(G_2) = 3$, but $G_2$ does not even have the unique OSJ cover graph. Indeed, $G_2$ has two OSJ covers. Both covers have components 
$$ \mc K_1 = \{1, 2\}, \mc K_2 = \{1, 4\}, \mc K_3 = \{1, 2, 3\},$$
$\C_2=\{\mc K_1 \vee_{[4]} \mc K_1, \mc K_1 \vee_{[4]} \mc K_2, \mc K_2 \vee_{[4]} \mc K_3\}$, and
$\C_3=\{\mc K_1 \vee_{[4]} \mc K_1, \mc K_2 \vee_{[4]} \mc K_3\}$. Note that $G(\C_2)$ is a path $P_3$ with a loop on one pendant vertex, and
$G(\C_3)$ is $K_1^{\ell} \cup P_2$. 

Finally, let $G_3 = K_3 \vee K_1^{\ell}$ be a graph on $4$ vertices. Denote the vertices of $K_3$ by $1$, $2$, $3$ and the vertex with the loop by $4$. Again $\spr(G_3)=3$. This time $G_3$ has the unique OSJ cover graph $K_3$, but OSJ cover is not unique nor essentially unique. Let
$$ \mc K_1 = \{1, 4\}, \mc K_2 = \{2, 4\}, \mc K_3 = \{3, 4\}, \mc K'_3 = \{3\}.$$
The covers $\C_4=\{\mc K_1 \vee_{[4]} \mc K_2, \mc K_2 \vee_{[4]} \mc K_3,\mc K_3 \vee_{[4]} \mc K_1\}$ and  $\C_5=\{\mc K_1 \vee_{[4]} \mc K_2, \mc K_2 \vee_{[4]} \mc K'_3,\mc K'_3 \vee_{[4]} \mc K_1\}$ are not isomorphic since the cardinalities of their components do not match.  

In Example \ref{ex:K6unique} we will see that $K_6$ has essentially unique but not unique OSJ cover.
\end{example}

\subsection{Operations on graphs that preserve  $\spr(G)$}

Observations so far make it clear, that if $G$ has two vertices with the same set of neighbours, then removing one of those vertices will not change $\spr(G)$. This gives us an operation on graphs, whose effect on SNT-rank is easily understood. To make this observation more precise, we need the definition below. 

\begin{definition}
Let $G$ be a graph and $v \in V(G)$. Then 
$$N_G(v):=\{w; \{v,w\} \in E(G)\}$$
is called the \emph{neighbourhood} of $v$.  
\end{definition}

Note that $v \in N_G(v)$ precisely when $v$ has a loop in $G$. 

\begin{definition}
 Two vertices $v, w \in V(G)$ are called \emph{twins} if  $N_G(v)=N_G(w)$. A graph $G$ is \emph{twin-free}, if no two pairs of vertices in $V(G)$ are twins.  

By $F_{tw}(G)$ we denote the biggest twin free sub-graph of 
$G$. Note that   $F_{tw}(G)$  is obtained from $G$ by removing all but one vertex from every set of twins in $G$.  
\end{definition}

\begin{definition}
Let $A \in \mc S_n^+$. The graph $F_{tw}(G(A))$ is called \emph{the twin-free graph} of $A$.
\emph{The twin-free pattern matrix} of $A$, denoted by  $F_{tw}(\sign(A))$, is the adjacency matrix of $F_{tw}(G(A))$ and can be obtained from $\sign(A)$ by removing any duplicate rows and columns from $\sign(A)$.
\end{definition}

With this new terminology we can restate our earlier observation. 

\begin{proposition}\label{prop:twin}
Let $G$ be a graph. Then $\spr(G)=\spr(F_{tw}(G))$. Moreover, $G$ has unique OSJ cover if and only if $F_{tw}(G)$ has.
\end{proposition}

\begin{proof}
Let $G$ be a graph, $v \in V(G)$, and $\widehat G$ a graph obtained from $G$ by duplicating the vertex $v$ into vertices $v_1$ and $v_2$. Hence, $V(\widehat G)=(V(G)\setminus\{v\})\cup \{v_1,v_2\}$, and $N_{\widehat G}(v_i)$, $i=1,2$, is equal to $N_{G}(v)$ if $v \not\in N_{G}(v)$, and is equal to $(N_{G}(v)\setminus\{v\})\cup\{v_1,v_2\}$ otherwise. 

A set-join cover of $\widehat G$ can be obtained from a set-join cover $\C$ of $G$ by replacing $v$ in each of the components of $\C$ by $v_1$ and $v_2$. Note that two distinct set-join covers of $G$ result two distinct set-join covers of $\widehat G$ in this process. Hence, $\spr(G)\geq \spr(\widehat G)$ and unique OSJ cover of $\widehat G$ implies unique OSJ cover of $G$.

In order to prove that unique OSJ cover of $G$ also implies unique OSJ cover of $\widehat G$, let us assume that $G$ has unique OSJ cover, and let ${\widehat \C}$ be some OSJ cover of $\widehat G$.
We claim that for $\mc K \in V({\widehat \C})$, $v_1 \in \mc K$ if and only if $v_2 \in \mc K$. If this is not true, then one set-join cover $\C_1$ of $G$ can be obtained from ${\widehat \C}$ by removing $v_2$ and replacing $v_1$ in each of the components of ${\widehat \C}$ by $v$, and a different set-join cover $\C_2$ of $G$ can be obtained from ${\widehat \C}$ by removing $v_1$ and replacing $v_2$ in each of the components of ${\widehat \C}$ by $v$. In particular, this proves that every OSJ cover of $\widehat G$ can be obtained from some OSJ cover of $G$ by the process outlined above. Since the OSJ cover of $G$ is by our assumption unique this establishes the uniqueness of OSJ cover of $\widehat G$. 
\end{proof}

\begin{lemma}
If $\C$ is an optimal set-join cover for a graph $G$, then $G(\C)$ is twin-free. 
\end{lemma}
 
\begin{proof} 
The assertion follows directly from Propositions \ref{prop:OSJgraph} and \ref{prop:twin}.
\end{proof}

\begin{remark}
Suppose that $G(\C)$ contains twin vertices $\mc K_1$ and $\mc K_2$. Let $\mc K_3 = \mc K_1 \cup \mc K_2$ and 
\begin{align*}
\C_1 &= \{\mc K \vee_{V(G)} \mc K'; \mc K, \mc K' \in (V(\C)\setminus\{\mc K_1,\mc K_2\}), \mc K \vee_{V(G)} \mc K' \in \C\}, \\  
\C_2 &= \{\mc K_3 \vee_{V(G)} \mc K'; \mc K' \in (V(\C)\setminus\{\mc K_1,\mc K_2\}), \mc K_1 \vee_{V(G)} \mc K' \in \C\}.  \\
\C_3 &= \{\mc K_3 \vee_{V(G)} \mc K_3\}.  
\end{align*}
If $\mc K_1 \vee_{V(G)} \mc K_2 \in \C$ let $\widehat\C = \C_1 \cup \C_2 \cup \C_3$, otherwise let $\widehat\C = \C_1 \cup \C_2$.
It is straightforward to check that $\widehat\C$ is set-join cover of $G$ with $V(\widehat\C) = (V(\C)\setminus\{\mc K_1,\mc K_2\})\cup\{\mc K_3\}$, so $\C$ is not an optimal set-join cover of $G$.    
\end{remark}

\begin{proposition}\label{prop:vK1}
Let $G$ be a graph and let $\widehat{G}=G\vee K_1^{\ell}$. Then:
 $$\spr(\widehat{G})=
    \begin{cases}
    \spr(G); & \text{ if } |N_G(v)| \geq 1 \text{ for all } v \in V(G),\\ 
    \spr(H)+2; & \text{ if }  G=H\cup t K_1.
    \end{cases}
    $$
Moreover, if $G$ is a graph without isolated vertices, then $G$ has unique OSJ cover graph if and only if $\widehat{G}$ has one. 
\end{proposition}

\begin{proof}
Since any $A \in \mc S_+(G \vee K_1^{\ell})$ has  $A_1 \in \mc S_+(G)$ as a principal sub-matrix $\spr(G) \leq \spr(G \vee K_1^{\ell})$ holds.

Assume that $G$ has no isolated vertices. Let $V(G)=[n]$ and $V(\widehat{G})=[n+1]$. Replacing every component $\mc K $ in a set-join cover of $G$ by $\widehat{\mc K}:= \mc K\cup \{n+1\}$ produces a set-join cover of $\widehat{G}$. Conversely, replacing every component $\mc K$ in a set-join cover of  $\widehat{G}$ by $\mc K':= \mc K\setminus \{n+1\}$ produces a set-join cover of $G$. Observe that $\{n+1\}$ is not a component of any OSJ cover of $\widehat{G}$. Indeed, if a cover $\widehat{\C}$ of  $\widehat{G}$ contains $\{n+1\}$ as a component, then we can construct a set join cover $\widehat{\C}'$ of  $\widehat{G}$ satisfying $|\widehat{\C}'|=|\widehat{\C}|-1$   by  removing any set-joins involving $\{n+1\}$ and adding $n+1$ to all other components of $\C$. Since the process outlined above preserves the order and the graphs of the corresponding OSJ covers, the statement for graphs without isolated vertices follows. (We note in passing that $G$ may have unique OSJ cover, but $\widehat{G}$ doesn't, since the vertex $n+1$ is not necessarily contained in all the components of OSJ cover of $G \vee K_1^{\ell}$.)

Now, let $G=H \cup t K_1$, where we assume that $H$ does not contain any isolated vertices (without loops). Let $V(G)=V(H) \cup \mc I=[n]$, and $V(\widehat{G})=[n+1]$ as before.  Defining $\widehat{\mc K}$ as above, produces from a set-join cover $\C$ of $G$ a set-join cover $\widehat{\C}$ of $(G \vee K_1^{\ell}) \cup t K_1$, since the elements of $\mc I$ are not contained in any component in this construction.
Adding $\mc I \vee_{[n+1]}\{n+1\}$ to $\widehat{\C}$ results in a set-join cover of $\widehat G$, and introduces $2$ new components. We conclude that $\spr(\widehat{G}) \leq \spr(H)+2$. 

Finally, let ${\C'}$ be an OSJ cover  of $\widehat{G}$. Since the elements of $\mc I$ are connected only to $\{n+1\}$ in $\widehat{G}$, we observe that ${\C'}$ has to contain $\{\mc K_0 \vee_{[n+1]} \{n+1\}\}$, where $\mc I \subseteq \mc K_0$, and $\mc K_0$ is not a component of any other set-join in ${\C'}$. Let $\C''$ be obtained from $\C'$ by removing all set-joins with a component $\{n+1\}$. Note that $|\C''| \leq |\C'|-2$, since $V({\C''})$ does not contain $\{n+1\}$ nor $\mc K_0$. Replacing every $\mc K \vee_{V(\widehat G)}\mc L \in \C''$ by $\mc K' \vee_{V(H)}\mc L'$, where $\mc K' :=\mc K\setminus \{n+1\}$  as above, gives us a set-join cover of $H$, proving   $\spr(\widehat{G})=\spr(H)+2$. 
\end{proof}

\begin{remark}
Note that $G = K_2 \cup K_1$ has unique OSJ cover graph, but $\widehat{G}=G\vee K_1^{\ell}$ doesn't.  
\end{remark}

From Propositions \ref{prop:union}, \ref{prop:twin} and \ref{prop:vK1} we see that in order to understand SNT-rank of graphs, we can from now on consider only connected twin-free graphs that are not of the form $G \vee K_1^{\ell}$.

\begin{example}
Threshold graphs is a family of graphs (without loops) that can be constructed by repeating two operations, adding an isolated vertex ($G\cup K_1$), and joining a vertex ($G \vee K_1$). Here we extend this definition to \emph{threshold graphs with loops} to be all graphs that can be constructed by repeating the following two operations on $G$:
\begin{itemize}
    \item $G \cup K_1$
    \item $G \vee K_1^{\ell}$
\end{itemize}
To obtain a twin free graph the two operations have to alternate. Considering only connected twin-free graphs resulting from this process, we get the following sequence of graphs: $$T_1:=K_1 \vee K_1^{\ell},\,  T_{i+1}:=(T_i \cup K_1)\vee K_1^{\ell}.$$  We have $\spr(T_i)=2i$ by Proposition \ref{prop:vK1}. This is not surprising, as it is not difficult to see that every matrix $A \in \mc S^+(T_i)$ has $\rk(A)=2i$. \end{example}

\begin{remark}
Let $G$ be a graph with a cut edge $\{u_1,u_2\}$, so that $G$ with this edge removed is equal to $G_1 \cup G_2$, where $u_i \in V(G_i)$, $i=1,2$. Then:
$$\spr(G_1)+\spr(G_2) \leq \spr(G) \leq \spr(G_1)+\spr(G_2)+2.$$
 Moreover, $\spr(G) = \spr(G_1)+\spr(G_2)$ if and only if for $i=1,2$ there exist OSJ covers $\C_i$ with $\{u_i\} \in V(\C_i)$, and $\spr(G) = \spr(G_1)+\spr(G_2)+1$ if such cover exists for either $i=1$ or $i=2$, but not both.   
\end{remark}

\section{Trees and cycles without loops}\label{sec:without-loops}

Let $G$ be a graph without loops that does not contain any four cycles. Then a set-join cover of $G$ can contain only elements of the form $\mc K \vee_{V(G)} \{v\}$ for some $v \in V(G)$, since all set-joins $\mc K \vee_{V(G)} \mc L$ with $|\mc K| \geq 2$ and $|\mc L|\geq 2$ contain a four cycle. Hence, for graphs without loops and four cycles, a set-join cover is equivalent to an edge star cover, as defined below. 

\begin{definition}
Let $G$ be a simple graph without loops. A family of simple stars $\{S_1, S_2, ...,$ $S_k\}$ is \emph{an edge star cover} of $G$ if $E(G) = \cup_{i=1}^k E(S_i)$. The \emph{edge star cover number} $\starc(G)$ of $G$ is the minimal number of stars in any edge star cover of $G$. 
\end{definition}

Clearly, $\spr(G)\leq 2 \starc(G)$ for any graph $G$, and we will show that this is an equality for trees. For any tree $T$ it is known that  $\rk(A)$ equals twice the matching number of $T$ for all $A \in \mc S^+(T)$, \cite{MR1861120}.
Since for a tree $T$ the matching number equals $\starc(T)$, the inequality $\spr(T)\geq 2 \starc(T)$ clearly holds. In this work we use the edge star cover number
due to its immediate connection to set-join covers. We summarise this observation in the next proposition. 

\begin{proposition}\label{th:tree}
Let $T$ be a forest (without loops) and $A \in \mc S^+(T)$. Then 
$\rk(A)=\spr(A)=
2 \starc(T)$. In particular, $\spr(T)=2 \starc(T)$.
\end{proposition}

In the next lemma we see that if a graph $G$ contains a leaf, then any set-join cover has to contain at least one element of the form $\mc K \vee \{v\}$ for some $v \in V(G)$, i.e. it contains at least one star.

\begin{lemma}\label{lem:leaf}
Let $G$ be a graph and let $L(G)\subset V(G)$ be the set of all leaves without a loop in $G$. Let $\ell \in L(G)$, $w \in V(G)$ its unique neighbour, and $G'$ the graph obtained from $G$ by deleting all edges $\{w,v'\}$ with $v' \in N_G(w)$ (and all singletons that result after this deletion). Then $\spr(G)=\spr(G')+2$ and for any optimal set-join cover $\C$ of $G$, we have $\{w\}\vee_{V(G)} \mathcal N \in \C$, where $N_G(w)\cap L(G) \subseteq \mathcal N \subseteq N_G(w)$. 
\end{lemma}

\begin{proof}
Let $\C$ be an optimal  set-join cover of $G$. Since $N_G(\ell)=\{w\}$, there exists $\mc N \subseteq N_G(w)$ with $\ell \in \mc N$, so that $\{w\} \vee_{V(G)} \mc N \in \C$. If there exists $\ell' \in (N_G(w)\cap L(G))\setminus \mathcal{N}$, then we also have $\{w\} \vee_{V(G)} \mc N' \in \C$ for some $\mc N'$ with $\ell' \in \mc N'$. Since $\C_1:=(\C\cup(\{w\}\vee N_G(w)))\setminus \{\{w\} \vee_{V(G)} \mc N,\{w\} \vee_{V(G)} \mc N'\}$ is a set-join cover of $G$ with $|\C_1|\leq |\C|-1$, we get a contradiction with the assumption that $\C$ is optimal. 
Hence, $N_G(w)\cap L(G) \subseteq \mathcal N$. 

 Observe that $\C':=\C\setminus \{\{w\} \vee_{V(G)} \mc N\}$ is a set-join cover for $G' \cup tK_1$ for some $t \in \mathbb{N}$. Since, $\spr(G')=\spr(G' \cup tK_1)$, this implies $\spr(G')\leq \spr(G)+2$. 
On the other hand, $\C=\{\{w\} \vee_{V(G)}N_G(w)\} \cup \C'$ is a set-join cover of $G$  for any set-join cover $\C'$ of $G'$. Hence, $\spr(G)\leq \spr(G')+2$, as desired.
\end{proof}

\begin{example}\label{ex:paths}
For paths we have  $\spr(P_{2k})=\spr(P_{2k+1})=2k$ by inductive application of Lemma \ref{lem:leaf}.
\end{example}

In the result below we resolve the question of uniqueness of OSJ covers  for trees without loops. 

\begin{theorem}\label{thm:Tcover}
Let $T$ be a tree with $|V(T)| \geq 3$. Then $T$ has unique OSJ cover if and only if the distance between any two leaves in $T$ is even.
\end{theorem}
 
\begin{proof}
Assume the distance between any two leaves in $T$ is even, and let $V_0$ be the set of vertices in $T$ at an odd distance to all leaves. An inductive application of Lemma \ref{lem:leaf} shows that
 $$\C_{0}:=\{\{v\} \vee N_T(v); v \in V_0 \}$$ is an OSJ cover of $T$. Let $\C'$ be an OSJ cover of $T$ and $V'$ the set of all central vertices of stars in $\C'$. Then $|\C'|=2|V'|=2|V_0|$. Since for any $v \in V(T)\setminus V'$ there exists $w \in V'\cap N_T(v)$, we necessarily have $V'=V_0$.  Since no two vertices in $V_0$ are connected in $T$, $\C_{0}$ is the only collection of stars with central vertices $V_{0}$ that covers $T$. We conclude that the cover is unique. We further remark, that the unique OSJ cover graph of $T$ is $|V_0|P_2$. 

Assume now, that there are two leaves $\ell_1$ and $\ell_2$ in a tree $T$ at odd distance $d$, $d \geq 3$. Observe that in any set-join cover $\C$ of $T$ at least every second vertex of the path between $\ell_1$ and $\ell_2$ is a central vertex of a star in that cover, hence a component in $\C$. Since $d$ is odd, we have $\{v_1\}\vee \mc S_1$ and $\{v_2\}\vee \mc S_2$ in $\C$, such that $v_1$ and $v_2$ are neighbours in $T$, and at least one of the inclusions $v_1 \in \mc S_2$ and $v_2 \in \mc S_1$ holds. Assume now that $\C$ is an OSJ cover, and without loss of generality that $v_1 \in \mc S_2$. Then we get a different OSJ cover of $T$ by replacing $\{v_1\}\vee \mc S_1$ and $\{v_2\}\vee \mc S_2$ by $\{v_1\}\vee (\mc S_1\cup \{v_2\})$ and $\{v_2\}\vee (\mc S_2\setminus\{v_1\})$.
\end{proof}

\begin{example}
Note that the only tree with $|V(T)| =2$ is a path $P_2$ and it has the unique OSJ cover. Moreover, a path $P_n$ has unique OSJ cover if and only if $n$ is odd or $n=2$, see Theorem \ref{thm:Tcover}.
\end{example}

For integers $t\geq 3$ and $k_i \geq 1$, $i \in [t]$, we denote by $\mathrm{star}(k_1, k_2, ..., k_t)$ the graph with the central vertex $v$ and $t$ arms, where each arm is a path on $k_i + 1$ vertices and one of its end-vertices is $v$. Hence,  $\mathrm{star}(k_1, k_2, ..., k_t)$ is a generalised star with $\sum_{i=1}^t k_i + 1$ vertices. If $k_i = 1$ for all $i$, then $\mathrm{star}(1,1, ..., 1)=K_{1,t}$ is a star graph with $t$ leaves. 

\begin{example}\label{pr:star}
Let $G=\mathrm{star}(k_1, k_2, ..., k_t)$, $t\geq 3$,  be a generalised star. Then  
$$\spr(G)=
    \begin{cases}
    |V(G)| + 1 - |\{k_i; k_i \text{ odd}\}|; & \text{if at least one } k_i \text{ is odd},\\ 
    |V(G)| - 1; & \text{if all } k_i \text{ are even}.
    \end{cases}
    $$

We will prove this claim by induction on $$\psi(G):=\sum_{i=1}^t k_i-t.$$ We first assume $\psi(G)=0$. This implies $G = G_0(t):=\mathrm{star}(1,1,\ldots,1)$, i.e. a star with $t$ leaves. Clearly, $\spr(G_0(t))=2$, and the claim holds. For  $\psi(G)=1$, we have $G = G_1(t):=\mathrm{star}(2,1,\ldots,1)$, i.e. a generalised star with $t$ arms. Applying Lemma  \ref{lem:leaf} for the leaf on the long arm, we get: 
$$\spr(G_1(3))=\spr(P_3)+2=4$$ 
by Example \ref{ex:paths}, 
and for $t>3$:
$$\spr(G_1(t))=\spr(G_0(t-1))+2=4.$$ 
Since $G_1(t)$ has $t+2$ vertices and $t-1$ arms of odd length, this establishes the base of induction. 

Now assume the claim holds for all generalised stars $G'=\mathrm{star}(k_1', k_2', ..., k_{t'}')$ with $\psi(G')\leq \psi$, and let $G:=\mathrm{star}(k_1, k_2, ..., k_t)$ with $\psi(G)=\psi+1\geq 2$. Without loss of generality we can assume $k_1\geq 2$. By Lemma \ref{lem:leaf} we have $\spr(G)=\spr(G')+2$, where:
\begin{itemize}
\item $G':=\mathrm{star}(k_2,\ldots,k_t)$ with $\psi(G')=\psi(G)-1\geq 1$ for $k_1=2$,
\item $G':=\mathrm{star}(k_1-2,k_2,\ldots,k_t)$  with $\psi(G')=\psi(G)-2$ for $k_1 \geq 3$. 
\end{itemize}
If $k_1=2$ and $t=3$, then $G'=P_{k_2+k_t+1}$, and the claim holds for $G$  by Example \ref{ex:paths}. In all other cases $G'$ is again a generalised star, so we can use the induction hypothesis. Noticing that $V(G)-V(G')=2$ and that $G$ and  $G'$ have same number of arms of odd length, we establish the claim.

The graph $G$ has unique OSJ cover if and only if all $k_i$'s are odd or all even by Theorem \ref{thm:Tcover}.
\end{example}

A vertex $v$ is a \emph{cut} vertex of a connected graph $G$, if by removing $v$ and all edges $\{v,w\}$, $w \in N_G(v)$, from $G$ we get a disconnected graph that we will denote by $G\setminus \{v\}$. (We allow for a cut vertex to have a loop.) 

\begin{proposition}\label{pr:cut}
Let $v$ be a cut vertex of a connected graph $G$, where $k \geq 2$ and $G\setminus\{v\}=\cup_{i=1}^k G_i$ for some connected graphs $G_i$. Then $\spr(G) \leq 2+ \sum_{i=1}^k \spr(G_i)$. 
\end{proposition}

\begin{proof}
Let $\C_i$ be a set-join cover of $G_i$, $i \in [k]$. Then $$\C := \cup_{i=1}^k \C_i \cup \{\{v\} \vee_{V(G)} N_G(v)\}$$ is a set-join cover of $G$. Since $|\C| = 2 + \sum_{i=1}^k |\C_i|$, the claim follows. 
\end{proof}

\begin{example}
The inequality of Proposition \ref{pr:cut} is not always an equality, but it cannot be improved in general.
\begin{enumerate}
\item Let $G=\mathrm{star}(2,2,2)$. Then $\spr(G) = 6$ by Proposition \ref{pr:star} and $G\setminus\{v\}=3P_2$, so $2+ \sum_{i=1}^k \spr(G_i) = 8$.
\item Let $G$ be a star graph $G=\mathrm{star}(3,3,3)$. Then $\spr(G) = 8$ by Proposition \ref{pr:star} and $G\setminus\{v\}=3P_3$, so $2+ \sum_{i=1}^k \spr(G_i) = 8$. 
\end{enumerate}
\end{example}

We now consider SNT-rank of cycles.

\begin{proposition}\label{prop:cycle}
Let $n \ge 3, n\ne 4$. Then $\spr(C_n)=n$. Furthermore,  $\spr(C_4)=2$. 
\end{proposition}

\begin{proof}
If $n = 3$, then the determinant of any matrix $A \in \mc S^+(C_3)$  is positive, so $\rk(A) = 3$, thus $\spr(C_3) = 3$. If $n = 4$,  then 
$$ \C = \{\mc \{1,3\} \vee_{[4]} \mc \{2,4\}\}$$
is a set-join cover of $C_4$, so $\spr(C_4)=2$. Now let $n \ge 5$ and let $\C$ be a set-join cover of $C_n$. If  every component in $\C$ is a singleton, then $|\C| = n$. If $\C$ contains a component $\mc K_1$, which is not a singleton, then $\mc K_1 = N_{C_n}(v)$ for some $v\in V(C_n)$ and $\{v\} \vee_{[n]} \mc K_1 \in \C$. Since $n \ge 5$ the components $\{v\}$ and $\mc K_1$ don't appear anywhere else in $\C$. This implies that $\C \setminus \{\{v\} \vee_{[n]} \mc K_1\}$ is a set-join cover of $P_{n-1}$, so $|\C| \ge 2 + \spr(P_{n-1}) \ge n$, proving $\spr(C_n)=n$.
\end{proof}

\begin{remark}
Let $G$ be an  unicyclic graph, i.e. connected simple graph without loops containing exactly one cycle. Then $G$ is either a cycle or it contains at least one leaf. We can find the SNT-rank of $G$ by repeatedly applying Lemma \ref{lem:leaf} on a leaf. In each step we obtain a disjoint union of trees and at most one unicyclic graph. Eventually we obtain a disjoint union of trees and at most one cycle. For this graph the SNT-rank can be computed by Proposition \ref{th:tree} and Proposition \ref{prop:cycle}.
\end{remark}

\section{Complete graphs without loops}\label{sec:complete}

Let $\C$ be a set-join cover of $K_n$ with $V(K_n) = [n]$. If $\mc K \vee_{[n]} \mc L \in \C$, then $\mc K \cap \mc L = \emptyset$, since $K_n$ has no loops. On the other hand, for any pair $i,j \in [n],$ $i \neq j$, there exist
 $\mc K \vee_{[n]} \mc L \in \C$, so that $i \in \mc K$, $j \in \mc L$.  
 
 \begin{definition}
Let $\bf T$ be a $k$-tuple of subsets of $[n]$. If for any pair $i,j \in [n],$ $i \neq j$, there exists  $\mc K, \mc L \in \bf T$ so that $i \in \mc K$, $j \in \mc L$ and $\mc K \cap \mc L=\emptyset$, then $\bf T$ is called \emph{a separating cover} of $n$ elements with $k$ sets. 
  \end{definition}

Hence, if $\C$ is a set-join cover of $K_n$, then $V(\C)$ is a separating cover of $n$ elements with $|V(\C)|$ sets. To determine $\spr(K_n)$ for a given $n$, we need to find minimal separating cover of $n$ elements. This problem is known as Katona problem and has been solved independently by A. C. C. Yao, \cite{MR422036} and M. C. Cai \cite{MR732208}, hence the problem of determining $\spr(K_n)$ is resolved. In this section we expand on the work in \cite{MR422036}, by showing how various set-join covers of $K_n$ can be constructed, investigate when these covers are optimal, and when they are essentially unique.

\begin{theorem}[\cite{MR422036}]\label{th:Katona}
Let $n \in \mathbb{N}$, and let $s(n):=\min\{k:$ there is a separating cover on $n$ elements with $k$ sets$\}$. For all $n \geq 2$, 
\begin{equation}\label{eq:Katona}
s(n)=\begin{cases}
3i &\text{if }\, 2\cdot 3^{i-1} <n \leq 3^i \\
3i+1 &\text{if }\,  3^{i} <n \leq 4\cdot 3^{i-1} \\
3i+2 &\text{if }\, 4 \cdot 3^{i-1} <n \leq 2\cdot 3^i, \\
\end{cases}
\end{equation}
while $s(1)=0$.
\end{theorem}

\begin{corollary}
The SNT-rank of a complete graph is $\spr(K_n)=s(n)$, where $s(n)$ is defined in \eqref{eq:Katona}.
\end{corollary}

To give some intuition on how a set-join cover for $K_n$ can be constructed, we offer the following example. 

\begin{example} \label{ex:pq}
Suppose that $n \in \mathbb{N}$ is written as a product of two factors $n=p\cdot q$ with $p, q \ge 2$. Let $V(K_n) = [n]$, and define: 
\begin{align*}
 \mc L_i &= \{(i-1)p+s; s=1,\ldots,p\} \ \text{ for every }  i \in [q], \\
 \mc K_j &= \{j+tp; t=0,\ldots, q-1\}, \ \text{ for every }  j \in [p].
\end{align*}
Then
$$\C =  \{\mc L_i \vee_{[n]} \mc L_j, i, j \in [q], i \ne j\} \cup  \{\mc K_i \vee_{[n]} \mc K_j, i, j \in [p], i \ne j\}$$
 is a set-join of $K_n$. It follows that $\spr(K_n) \le p + q$.
\end{example}

A generalisation of this example leads to an OSJ cover of $K_n$. 

\begin{proposition}\label{lem:prod}
Let $n \in \mathbb{N}$ and  $q_i \in \mathbb{N}$, $i \in [t]$, satisfy $q_i \geq 2$ and $ n\leq \prod_{k=1}^t q_k$. Then there exists a set-join cover $\C$ of $K_n$ with $|\C|=\sum_{i=1}^t q_i$ and $G(\C)=\cup_{i=1}^t K_{q_i}$. 
Furthermore, 
\begin{equation}\label{eq:Katona_factor}
\spr(K_n)= \min\left\{\sum_{i=1}^t q_i; n\leq \prod_{i=1}^t q_i\right\}.
\end{equation}
\end{proposition}
 
\begin{proof} 
We prove the first claim by induction on $t$. If $t = 1$, then the claim is trivial. If $n \le m$, then $K_n$ is a subgraph of $K_m$ and a restriction of any set-join cover of $K_m$ to the elements of $K_n$ results in a set-join cover of $K_n$, in particular $\spr(K_n) \le \spr(K_m)$. This observation allows us to assume $n = \prod_{k=1}^t q_k$. The claim for $t=2$ is proved in Example \ref{ex:pq} and establishes the base of induction. Let $n_1 = \prod_{k=1}^{t-1} q_k$, $V(K_n) = [n]$ and $V(K_{n_1}) = [n_1]$. Let $\C$ be a set-join cover of $K_{n_1}$ with $|\C| = \sum_{k=1}^{t-1} q_k$. For every $\mc K \in V(\C)$ let 
$$  \mc K' = \cup_{i\in \mc K}\{i+tn_1;t=0,\ldots,q_t-1\} \subset [n]. $$
Furthermore, let
$$  \mc L_i =  \{(i-1)n_1+s;s=1,\ldots, n_1\} \ \text{ for every }  i \in [q_t]. $$
Then
$$\C' = \{\mc K'_1 \vee_{[n]} \mc K'_2; \mc K_1 \vee_{[n_1]} \mc K_2 \in \C\}$$
covers all the edges of $K_n$ within the sets $\mc L_1, \mc L_2, ..., \mc L_{q_t}$. So
$$\widehat{\C} = \C' \cup \{\mc L_i \vee_{[n]} \mc L_j, i, j \in [q_t], i \ne j\}$$
is a set-join cover of $K_n$ with $|\widehat{\C}| = \sum_{i=1}^t q_i$.

To complete the proof, we need to show $\spr(K_n) \ge \min\{\sum_{i=1}^t q_i; n\leq \prod_{i=1}^t q_i\}$. Let us fix $n$ and define $m = \max\{t; s(t)=~ s(n)\}$. From Theorem \ref{th:Katona} we know that $m$ is the largest positive integer that satisfies $\spr(K_n) = \spr(K_m)$. In particular,  $m$ has one of the forms $m=3^i, m= 4\cdot 3^{i-1}$, or $m = 2\cdot 3^i$. In each case we can write $m=\prod_{k=1}^t q_k$, where $q_j \in \{ 2,3,4\}$ for all $j \in [t]$. It is straightforward to check that the equality $\spr(K_m) = \sum_{i=1}^t q_i$ holds in each of these cases.
\end{proof}

\begin{remark}\label{ex:Kn}
An optimal set-join cover of $K_6$ is given in Example \ref{ex:K6}. 
Following the construction outlined in the proof of Proposition \ref{lem:prod}, we can produce examples of optimal set-join covers of $K_n$ for every $n \ge 2$. 

If $n=2, 3, 4, 5$, then $\spr(K_n) =n$, and the set-join cover $\C$, where all the components of $\C$ are singletons and $G(\C) = K_n$, is an OSJ cover of $K_n$. Note that for $n=2$ and for $n=3$ these OSJ covers are unique. For $n=4$ this set-join cover is not unique, since the factorization $4=2^2$ gives rise to an OSJ cover $\C'$ of $K_4$ with $G(\C')=2K_2$. In the case $n=5$ this set-join cover is also not unique, since the factorization $6=2\cdot 3$ gives rise to an OSJ cover $\C'$ of $K_5$ with $G(\C')=K_2 \cup K_3$.

Suppose that $n \ge 6$ and let $m = \max\{t; s(t) = s(n)\}$. An inductive construction outlined in the proof of Proposition \ref{lem:prod} gives an OSJ cover $\C$ of $K_m$, which we can restrict to $K_n$ to obtain an OSJ cover of $K_n$. 
Let us look at some properties of an OSJ cover $\C$ constructed in this way, for a few choices of $n$:
\begin{itemize}
\item For $n=3^i$, all components of $\C$ have cardinally $3^{i-1}$, and for each $\mc K \in \C$ there exist two other components $\mc K',\mc K'' \in V(\C)$, so that $\mc K, \mc K'$ and $\mc K''$ and pairwise disjoint.
\item For $n=3^i-1$, $2i$ components have cardinality $3^{i-1}$, and $i$ components have cardinality $3^{i-1}-1$. For each $\mc K \in \C$ there exist two other components $\mc K',\mc K'' \in V(\C)$, so that $\mc K, \mc K'$ and $\mc K''$ and pairwise disjoint.
\item For $n=2\cdot 3^i$,  the largest two components $\mc K, \mc K' \in V(\C)$ are disjoint and have cardinality $3^i$. All other components have cardinality $2\cdot 3^{i-1}$. 
\item For $n=2\cdot 3^i-1$,  the largest two components in $V(\C)$ are disjoint, one has $3^i$ and the other $3^{i}-1$ elements. There are $2i$ components of cardinality $2 \cdot 3^{i-1}$, and $i$ components of cardinality $2 \cdot 3^{i-1}-1$.
\end{itemize}
\end{remark}

Having constructed an OSJ cover of $K_n$ for each $n \in \mathbb{N}$, we want to consider what properties any OSJ cover $\C$ of $K_n$ has to have. In particular, we consider the components of minimal and maximal cardinality. 

\begin{lemma}\label{lem:min}
Let $m,n \in\mathbb{N}$, $n< m$, $\spr(K_n)=\spr(K_m)$, and let $\C$ be an OSJ cover of $K_m$. Then any component $\mc K \in V(\C)$ satisfies $|\mc K|>m-n$.
\end{lemma}
\begin{proof}
Let $\C$ be an OSJ cover of $K_m$, and $\mc K \in V(\C)$. Let $\C'$ be obtained from $\C$ by removing $\mc K$, and elements from $\mc K$ from all components of $\C$. Then $\C'$ is a set-join cover of $K_{m-|\mc K|}$ with $|\C'|\leq |\C|-1$. From $\spr(K_n)=\spr(K_m)$ we get $m-|\mc K|<n$, as required.
\end{proof}

\begin{corollary}
Let $\C$ be an OSJ cover of $K_n$ and $\mc K \in V(\C)$.
\begin{enumerate}
\item If $n=3^i$, then $|\mc K|\geq 3^{i-1}$. 
\item If $n=4 \cdot 3^i$, then $|\mc K|\geq 3^i$. 
\item If $n=2 \cdot 3^i$, then $|\mc K|\geq 2\cdot 3^{i-1}$. 
\end{enumerate}
\end{corollary}

\begin{proof}
Let $n=3^i$, then $\spr(K_{3^i})=\spr(K_{2 \cdot 3^{i-1}+1})$ by Theorem \ref{th:Katona}. Hence, $|\mc K|>3^{i-1}-1$ by Lemma \ref{lem:min}. The other two items are proved in the same way. 
\end{proof}

\begin{proposition} \label{pr:notunique}
Suppose that $n \in \mathbb{N}$ satisfies one of the following conditions: $n\in \{4,5,7,8\}$, or $3^i+1 \le n \leq 2\cdot 3^i-2$ for some $i\ge 2$, or $2\cdot 3^i+1 \le n \leq 3^{i+1}-2$ for some $i\ge 2$. Then an OSJ cover of $K_n$ is not essentially unique.
\end{proposition}

\begin{proof}
In Remark \ref{ex:Kn} we have already seen that an OSJ cover of $K_n$ is not essentially unique for $n \in \{4, 5\}$. $K_8$ has two essentially different OSJ covers: one arising form the factorization $8 = 2\cdot 4$, and the other form $8 \le 3^2$. In the first case the cover has the graph $K_2 \cup K_4$, and in the second case the cover has the graph $2K_3$. Since $s(8) = 6$ both covers are optimal. Restrictions of those covers to the elements of $K_7$ gives two OSJ covers of $K_7$. 

Suppose that $n+2\le m$ and $s(n) = s(m)$. We will show that in this case an OSJ cover of $K_n$ is not essentially unique. Let $\C$ be an OSJ cover for $K_m$ and let $\mc K \in V(\C)$ be a component of $\C$ of minimal cardinality. Then $|\mc K| > m-n\geq 2$, by Lemma \ref{lem:min}. Now, choose $m-n$ vertices in $\mc K$ and remove them from $K_m$ and $\C$ to obtain a set-join cover $\C'$ of $K_n$. The set-join cover $\C'$ contains a component of cardinality $|\mc K| - (m-n)$. Next, choose $m-n$ vertices in $K_m$ so that they do not all belong to the same component of $\C$. This is possible since $m-n \ge 2$, and $m \ge 9$, so $m \ge |V(\C)|+3$. Remove them from $K_m$ and $\C$ to obtain a set-join cover $\C''$ of $K_n$. The set-join cover $\C''$ does not contain a component of cardinality $|\mc K| - (m-n)$, so it is essentially different from $\C'$.

For $n=4\cdot3^i=2^2\cdot3^i$ the construction in Proposition \ref{lem:prod} gives rise to two set-join covers $\C_1$ and $\C_2$, one with the graph $K_4 \cup iK_3$, and one with the graph $2K_2 \cup iK_3$. Both covers are optimal, proving that in this case OSJ cover is not unique. Now let $m$ satisfy $m < n$ and $s(m) = s(n)$. By restricting $\C_1$ and $\C_2$ to $[m]$, we get two OSJ covers of $K_m$ with different graphs, proving that OSJ covers of $K_m$ are not essentially unique. 
\end{proof}

Before we move to more general cases, we show
in the example below that $K_6$ has essentially unique OSJ cover. 

\begin{example}\label{ex:K6unique}
 Let $V(K_6) = [6]$ and let $\C$, with $V(\C) = \{\mc K_i; i=1,\ldots, 5\}$, be an OSJ cover of $K_6$. In addition, we assume $|\mc K_1|\geq |\mc K_i|$, $i=2,3,4,5$, and $\mc K_1 \vee_{[6]} \mc K_2 \in \C$. 

 Next we show that $|\mc K_1|=3$. Since $\mc K_1$ and $\mc K_2$ are disjoint, the set-joins of $\C$ with components $\mc K_3$, $\mc K_4$,  and $\mc K_5$, need to cover all the edges between the elements of $\mc K_1$. In particular, the set $V_1=\{\mc K_3 \cap \mc K_1, \mc K_4 \cap \mc K_1, \mc K_5 \cap \mc K_1\}$ needs to contain a component set of a set-cover for a complete graph with the vertex set  $\mc K_1$. Since $s(4) = 4$, this excludes  $|\mc K_1| \ge 4$. On the other hand, $s(6)=s(5)$ implies that $V(\C)$ doesn't contain singletons by Lemma \ref{lem:min}. If $|\mc K_i| = 2$ for all $i \in[5]$, then we may without loss of generality assume that $\mc K_1= \{1, 2\}$ and  $\mc K_2 \subset \{3, 4, 5, 6\}$. Since the edge $\{1,2\}$ needs to be covered, we may also assume that $\mc K_3= \{1, 3\}$ and $\mc K_4 = \{2, 4\}$. Similarly, since the edge $\{5,6\}$ needs to be covered, we may assume $5 \in \mc K_2$, $6 \in \mc K_5$ and $\mc K_2 \cap \mc K_5=\emptyset$. Now we have $j \in \mc K_2$ for $j=3$ or $j=4$. In either case, the edge $\{j,5\}$ cannot be covered by the components $\mc K_i$, $i=1,\ldots,5$. 

At this point we can assume $\mc K_1=\{1,2,3\}$, and deduce that $V_1$ is a component set of an OSJ cover for the complete graph on $\mc K_1$, hence $V_1=\{\{1\}, \{2\}, \{3\}\}$ and $\{\mc K_3 \vee_{[6]} \mc K_4,\mc K_4 \vee_{[6]} \mc K_5,\mc K_5 \vee_{[6]} \mc K_3\}\subset \C$. Now we know that $\mc K_3$, $\mc K_4$ and $\mc K_5$ are pairwise disjoint, and without loss of generality we may assume $\mc K_1\cap\mc K_3=\{1\}$, $\mc K_1\cap\mc K_4=\{2\}$ and $\mc K_1\cap\mc K_5=\{3\}$.

Next we argue that $\mc K_2=\{4,5,6\}$. Assume that $4 \not\in \mc K_2$, and note that $4$ can be contained in at most one of  $\mc K_3$, $\mc K_4$ and $\mc K_5$. Without loss of generality let $4 \in \mc K_3$. This leads to a contradiction, since the edge $\{1,4\}$ clearly cannot be covered. Arguing as above, we prove that $\{\mc K_3 \cap \mc K_2, \mc K_4 \cap \mc K_2, \mc K_5 \cap \mc K_2\}=\{\{4\},\{5\},\{6\}\}$. We may assume without loss of generality that $\mc K_3 = \{1, 4\}, \mc K_4 = \{2, 5\},$ and $\mc K_5 = \{3, 6\}$. We obtain that 
$$\C=\{\mc K_1 \vee_{[6]} \mc K_2, \mc K_3 \vee_{[6]} \mc K_4, \mc K_4 \vee_{[6]} \mc K_5,\mc K_5 \vee_{[6]} \mc K_3\}$$
and $G(\C) = K_2 \cup K_3$.

Notice that an OSJ cover of $K_6$ constructed above is not unique, since by permuting the vertices of $K_6$ we can obtain a different set-join cover of $K_6$.
\end{example}

To identify $K_n$ with  essentially unique OSJ covers, we need to extend the arguments seen in Example \ref{ex:K6unique}. To do so, we depend on a selection of observations, that we do not prove here, but that can be deduced from the proof of Theorem \ref{th:Katona} given in \cite[Ch. 18]{MR801393}. 

\begin{lemma}(\cite{MR422036, MR801393})\label{lem:math-gems}
Let $\C$ be an OSJ cover of $K_n$ for some  $n \geq 4$. Then
\begin{equation}\label{eq:Katon_ind}
s(n) = \min\left\{k + s\left(\left\lceil\frac{n}{k}\right\rceil\right); k = 2, 3, 4, 5\right\}.
\end{equation}
Let $\mc M_n$  be the set of all $k \in \{2,3,4,5\}$ for which the minimum is achieved, and $\mc K_1$ a component of $\C$ of maximal cardinality. Then there exists $k_0 \in \mc M_n$  so that: 
\begin{enumerate}
 \item $\mc K_1$ is disjoint with precisely $k_0-1$ components $\mc K_2,\ldots, \mc K_{k_0} \in V(\C)$. 
 \item The restriction of $\C$ to $\mc K_1$ is an OSJ cover of the complete graph on $\mc K_1$, and the restriction of $\C$ to $\mc S = [n] \setminus \cup_{j=2}^{k_0} \mc K_i$ is an OSJ cover of the complete graph on $\mc S$. 
 \item  $s(|\mc S|) = s(|\mc K_1|)= s(\lceil\frac{n}{k_0}\rceil) = s(n)-k_0$.     
\end{enumerate}
\end{lemma}

\begin{proposition} \label{pr:3nai}
Suppose that $n \in \mathbb{N}$, $n = 3^i$ or $n= 2\cdot 3^i$ for some $i\ge 1$. Then $K_n$ has essentially unique OSJ cover.
\end{proposition}

\begin{proof}
We will prove the claim by induction on $i$. Since we have already seen that $K_3$ and $K_6$ have essentially unique OSJ covers, the base of induction is established. We assume $i \geq 2$, and that the claim holds for  $n = 3^{i-1}$ and $n= 2\cdot 3^{i-1}$. Throughout the proof we also assume that $\C$ is an OSJ cover of $K_n$, and $\mc K_1 \in V(\C)$ is one of its components with maximal cardinality.

First we want to determine the parameter $k_0$ from Lemma \ref{lem:math-gems}.  Combining \eqref{eq:Katona} and \eqref{eq:Katon_ind} it is straightforward to see that  for $i \geq 2$ we have $\mc M_{3^i}=\{3\}$ and $\mc M_{2 \cdot3^i}=\{2,3\}$.
Suppose that for $n=2\cdot 3^{i}$ we have $k_0=3$. Applying Lemma \ref{lem:math-gems} we get $s(|\mc K_1|) =s(|\mc S|)= s(\lceil\frac{n}{3}\rceil) = s(2\cdot 3^{i-1}) = 3i-1$. This restricts $4\cdot 3^{i-2}<|\mc K_1| \leq |\mc S|\leq 2 \cdot 3^{i-1}$. From
\begin{align*}
|\mc K_2\cup \mc K_3|&\leq 2 |\mc K_1|\leq 4 \cdot 3^{i-1},\\
|\mc K_2\cup \mc K_3|&=n-|\mc S|\geq 2 \cdot 3^{i}-2 \cdot 3^{i-1}=4 \cdot 3^{i-1}
\end{align*}
we conclude that $|\mc K_2\cup \mc K_3|=4 \cdot 3^{i-1}$. Hence, $\mc K_2\cap \mc K_3=\emptyset$, and $|\mc K_i|=2 \cdot 3^{i-1}$ for $i=1,2,3$. Note that for $i=1,2,3$, $|\C[\mc K_i]|=|\C|-3$, hence $\mc C[\mc K_i]$ is an OSJ cover of $\mc K_i$, and by the induction hypothesis essentially unique. By Remark  \ref{ex:Kn}, $\C[\mc K_i]$ has two components of cardinality $3^{i-1}$, and all other components of cardinality $2\cdot 3^{i-2}$. 
Let $\mc K_4$ be one of the components of $\C$ with the largest intersection with $\mc K_1$. From $|\C[\mc K_i]|=|\C|-3$ for $i=2,3$, we deduce that  $\mc K_4 \cap \mc K_i\neq \emptyset$ for $i=2,3$. Then $|\mc K_4 \cap \mc K_1| = 3^{i-1}$ and $|\mc K_4 \cap \mc K_i| \ge  2\cdot 3^{i-2}$ for $i=2,3$. It follows that $|\mc K_4| \ge 3^{i-1} + 2\cdot 3^{i-2} + 2\cdot 3^{i-2} > 2\cdot 3^{i-1} = |\mc K_1|$, a contradiction with maximality of $|\mc K_1|$. 

Let $n = 3^i$ for $i \ge 2$, and hence $s(n) = 3i$ and $k_0=3$. From Lemma \ref{lem:math-gems} we get $s(|\mc K_1|)=s(|\mc S|)=3(i-1)$, implying $|\mc K_1|\leq|\mc S|\leq 3^{i-1}$. Hence,
\begin{align*}
|\mc K_2\cup \mc K_3|&\leq 2 |\mc K_1|\leq 2 \cdot 3^{i-1},\\
|\mc K_2\cup \mc K_3|&=n-|\mc S|\geq  3^{i}- 3^{i-1} =2 \cdot 3^{i-1}. 
\end{align*}
As above, we conclude that $\mc K_1$, $\mc K_2$, and $\mc K_3$ are pairwise disjoint, their union is $[n]$, and they all have $\frac{n}{3} = 3^{i-1}$ elements. Since we are allowing isomorphism of graphs, $\mc K_j$, $j \in [3]$,  can be taken to be any three subsets of $[n]$ that satisfy those conditions. 
Furthermore, $\C[\mc K_j]$, $j \in [3]$, is an OSJ cover of the complete graph on   $\mc K_j$ and by the induction hypothesis essentially unique. Since $\mc K_j$, $j \in [3]$, are pairwise disjoint and $\cup_{j \in [3]}\mc K_j=[n]$ this implies essential uniqueness of $\C$. 

The proof for $n = 2\cdot 3^i$, $i\geq 2$ with $s(n)=2 +3i$ and $k_0=2$ is very similar and we leave it to the reader. 
\end{proof}

\begin{remark}
With arguments, akin to the ones in the proof of Proposition \ref{pr:3nai},  it can be shown that for $n = 4 \cdot 3^i$ the graph $K_n$ has exactly two essentially different OSJ covers. In particular, the components of any OSJ cover of $K_n$ either satisfy $|\mc K_1| = \ldots = |\mc K_4| = 2\cdot 3^i$, $|\mc K_5| = \ldots = |\mc K_{3i+4}| = 4 \cdot 3^{i-1}$ or $|\mc K_1| = \ldots  = |\mc K_{3i}| = 4\cdot 3^{i-1}$, $|\mc K_{3i+1}| = \ldots = |\mc K_{3i+4}| = 3^i$.

Notice that in this case two possible values of $k_0 \in \mc M_n$ can be attained, namely in the first cover we have $k_0=2$ and in the second cover $k_0=3$.
\end{remark}

The following proposition covers the remaining cases with essentially unique OSJ cover. Only a sketch of the proof is given. The proof  with all the technical details included  would be rather long, and would not give significant new insights to the reader. 

\begin{proposition} \label{pr:3nai-1}
Let $n \in \mathbb{N}$ be of the form $n = 3^i - 1$ for some $i\ge 3$, or of the form $n= 2\cdot 3^i - 1$ for some $i\ge 2$. Then $K_n$ has essentially unique OSJ cover.
\end{proposition}

\noindent \emph{Sketch of the proof.} We claim that $K_n$ has essentially unique OSJ cover for every $n$ in the set 
$$\mc N = \{n \in \mathbb{N}, n = 3^i - 1, i\ge 3\} \cup \{n \in \mathbb{N}, n = 2\cdot 3^i - 1, i\ge 2\}.$$ This claim can be proved by induction, where the base of induction needs to be established for $n \in \{17,26\}$. (This is left to the reader.) Next we fix $n$, and assume that the claim holds for $m \in \mc N$, with $m < n$. We assume notation and definitions from Lemma \ref{lem:math-gems}. 

For $n = 3^i - 1, i\ge 3$, we have $\mc M_n=\{3\}$. Lemma \ref{lem:math-gems} implies $s(|\mc K_1|) =s(|\mc S|)=3 (i-1)$. From here we get:
$$|\mc K_2\cup \mc K_3|\geq n-|\mc S|\geq 2 \cdot 3^{i-1}-1.$$ 
Since $\mc K_1 \cap (\mc K_2\cup \mc K_3)=\emptyset$ and $\mc K_1 \cup \mc K_2 \cup \mc K_3 = [3^i - 1]$, we conclude that $|\mc K_1|=3^{i-1}$ and $|\mc K_2\cup \mc K_3|=2 \cdot 3^{i-1}-1$.
Note that $|\C[\mc K_2 \cup \mc K_3]|=|\mc C|-1=3i-1$, since $\mc K_1\cap (\mc K_2 \cup \mc K_3)=\emptyset$.  Hence $\C[\mc K_2 \cup \mc K_3]$ is an OSJ set-join cover of the complete graph on  $\mc K_2 \cup \mc K_3$, and it is by the induction hypothesis essentially unique. 
By Remark \ref{ex:Kn} we know that its largest components are disjoint (and of order $3^{i-1}$ and $3^{i-1}-1$). This implies that
$\mc K_2$ and $\mc K_3$ are disjoint, $|\mc K_1| = |\mc K_2| = 3^{i-1}$, and $|\mc K_3| = 3^{i-1} -1$. 
The claim is proved by noting that $\C[\mc K_j]$ is an essentially unique OSJ cover of the complete graph on $\mc K_j$ for $j\in [3]$.

For $n = 2\cdot 3^i - 1, i\ge 2$, we have $\mc M_n=\{2, 3\}$. The case $k_0=3$ can be excluded in a similar way, as was done in the proof of Proposition \ref{pr:3nai} for the case $n=2\cdot 3^i$. The rest of the proof of this case is very similar to the proof for the case $n=3^i - 1$ above, and it is left to the reader. 

\medskip

Propositions \ref{pr:notunique}, \ref{pr:3nai}, and \ref{pr:3nai-1} completely decide the question of essential uniqueness of OSJ covers of $K_n$. The result is summarized in the theorem below. 

\begin{theorem} Let $n \in \mathbb{N}$, $n \ge 2$. The graph $K_n$ has essentially unique OSJ cover
if and only if $n = 3^i$ for some $i\ge 1$, $n= 2\cdot 3^i$ for some $i\ge 1$, $n = 3^i -1$ for some $i\ge 3$, or $n= 2\cdot 3^i -1$ for some $i\ge 2$.
\end{theorem}

\begin{remark}
Equations \eqref{eq:Katona}, \eqref{eq:Katona_factor} and \eqref{eq:Katon_ind} express $\spr(K_n)$ in three different ways:
\begin{align*}
    \spr(K_n)&=\begin{cases}
3i &\text{if }\, 2\cdot 3^{i-1} <n \leq 3^i \\
3i+1 &\text{if }\,  3^{i} <n \leq 4\cdot 3^{i-1} \\
3i+2 &\text{if }\, 4 \cdot 3^{i-1} <n \leq 2\cdot 3^i \\
\end{cases} \\
&= \min\left\{\sum_{i=1}^t q_i; n\leq \prod_{i=1}^t q_i\right\} \\
&= \min\left\{k + \spr\left(K_{\left\lceil\frac{n}{k}\right\rceil}\right); k = 2, 3, 4, 5\right\}. 
\end{align*}
Through those expressions we have seen how OSJ covers can be constructed recursively. 
\end{remark}

\begin{remark}
In this section, an idea of constructing set-join covers from set-join covers of smaller graphs was repeatedly used. This idea can be formalised using co-normal products of graphs. Let $G$ and $H$ be simple graphs (with loops). The co-normal product of  $G$ and $H$, denoted by $G * H$, is the graph with the vertex set $V(G)\times V(H)$, and the edge set  
 $$E(G * H)=\left\{
 \{(g,h),(g',h')\}; \{g,g'\} \in E(G) \text{ or } \{h, h'\}\in E(H)\right\}.$$
Let $\C_G$
 and $\C_H$ be set-join covers of $G$ and $H$, respectively. Using $\C_G$ and $\C_H$ we can build a set-join cover $\C'$ of $G *H$ as follows. For every 
 $\mc K \in V(\C_G)$ define
 $\mc K':=\{(g,h); g\in \mc K, h\in V(H)\}$
and for every 
 $\mc L \in V(\C_H)$ define
 $\mc L':=\{(g,h); g\in V(G), h\in \mc L\}$. 
Then
$$\C':=\{\mc K_1'\vee_{V(G * H)} \mc K_2';\, \mc K_1\vee_{V(G)} \mc K_2 \in \C_G\} \cup \{\mc L_1'\vee_{V(G * H)} \mc L_2';\, \mc L_1\vee_{V(H)} \mc L_2 \in \C_H\}$$
is a set-join cover of $G * H$. In particular, 
\begin{equation}\label{eq:conormal}
\spr(G * H)\leq \spr(G)+\spr(H).
\end{equation}

Since $K_p * K_q$ is isomorphic to $K_{pq}$, we see from above, that the inequality \eqref{eq:conormal} can be strict (for example for $p=q=5$), but it cannot be improved in general (it is equality for example for $p=3^i$ and $q=3^j$). 
\end{remark}

\begin{remark}
If $A \in \mc S_n^+$, then its Boolean rank is the NMF-rank of the pattern graph of the matrix $A$  (see \cite{NMFGillis}):
$$ \rk_{01}(A) = \min\{\rk_+(B); B \in \mc S^+(G(A))\}.$$
It follows immediately from (\ref{eq:Katona}) that $\spr(K_n)$ behaves asymptotically as $\spr(K_n) \sim \frac{3}{\log 3} \log n$. On the other hand, if $A \in \mc S^+(K_n)$ its Boolean rank $\rk_{01}(A)$ is the minimal $k$ such that $n \le \binom{k}{\lfloor k/2 \rfloor}$, \cite{NMFGillis}. It follows that 
NMF-rank of such matrix is asymptotically bounded below by $\frac{1}{\log 2} \log n$. It means that the minimal SNT-rank and the minimal NMF-rank that can be achieved on the complete graphs have the same order of asymptotic behaviour, but the constant differs by factor $1.893$.
\end{remark}

\section{Concluding remarks} \label{sec:motivation}

Finding $\spr(G)$ for a given graph $G$ is well-defined combinatorial optimisation problem, that is of independent interest. In this section we highlight a few settings, where the theory of this paper naturally appears and provides relevant groundwork. 

The most immediate application of $\spr(G)$ is the study $\spr(A)$ and SN-Trifactorizations of $A \in \mc S_n^+$. Below we offer an example how this can be done for a given matrix. 

\begin{example}
The matrix
$$A=\left(
\begin{array}{cccc}
 4 & 1 & 1 & 4 \\
 1 & 1 & 2 & 0 \\
 1 & 2 & 0 & 3 \\
 4 & 0 & 3 & 1 \\
\end{array}
\right)$$
has rank $3$. Let $G:=G(A)$, and let $H$ be an induced subgraph of $G$ on the set $\{2, 3, 4\}$. Since any matrix in $\mc S^+(H)$ is invertible, we have $\spr(G) \ge 3$. Note also that $H$ has the unique OSJ cover $\C_H$, where each component in $V(\C_H)$ is a singleton. 

The set-join cover $\C_H$ can be extended to a set-join cover of $G$:
$$\C=\{\mc K_1 \vee_{[4]} \mc K_1, \mc K_1 \vee_{[4]} \mc K_2, \mc K_2 \vee_{[4]} \mc K_3,\mc K_3 \vee_{[4]} \mc K_3\},$$
where
$$ 2 \in \mc K_1, 3 \in \mc K_2, 4 \in \mc K_3,$$
and $1$ is an element of at least two of $\mc K_i$, $i=1,2,3$. This describes all possible set-join covers of $G$, and, in particular, shows that
 $\spr(G) = 3$. 
 
Suppose that $\spr(A) = 3$ and let $A = BCB^T$ be an optimal SN-Trifacto\-ri\-za\-tion. By observations above and Theorem \ref{th:main} the matrix $B$ has the form 
$$ B = \left(
\begin{array}{c}
 {\bf b}^T \\
 D \\
\end{array}
\right),$$
where ${\bf b} \in \R_+^3$ has at least two nonzero entries and $D$ is an invertible nonnegative diagonal matrix.  It follows that 
$$ A = BCB^T = \left(
\begin{array}{c}
 {\bf b}^T \\
 D \\
\end{array}\right) C \left(
\begin{array}{cc}
 {\bf b} & D \\
\end{array}\right) 
= \left(\begin{array}{cc}
 {\bf b}^TC{\bf b} & {\bf b}^TCD \\
 DC{\bf b}   & DCD \\
\end{array}\right),  $$ 
so 
$$ DCD = \left(
\begin{array}{ccc}
 1 & 2 & 0 \\
 2 & 0 & 3 \\
 0 & 3 & 1 \\
\end{array} \right) \text{ and }  DC{\bf b} = \left(
\begin{array}{c}
 1 \\
 1 \\
 4 \\
\end{array}
\right).$$ 
Thus 
$$ {\bf b} = D(DCD)^{-1}(DC{\bf b}) = D\left(\begin{array}{ccc}
 1 & 2 & 0 \\
 2 & 0 & 3 \\
 0 & 3 & 1 \\
\end{array} \right)^{-1}\left(
\begin{array}{c}
 1 \\
 1 \\
 4 \\
\end{array}
\right) = D \left(
\begin{array}{c}
 -1 \\
 1 \\
 1 \\
\end{array}
\right),$$
which means that vector ${\bf b}$ has a negative entry, a contradiction. It follows that $\spr(A) = 4$.
\end{example}

Given an $n \times n$ nonnegative symmetric matrix $A$ there are existing algorithms that given $k$ find matrices  $B \in \R_+^{n \times k}$ and $C \in \mc S_k^+$ that minimize $\norm{A-BCB^T}_F$,  \cite{MR2806386}. Using the theory developed in this paper it should be possible to extend such algorithms to find matrices  $B \in \R_+^{n \times k}$ and $C \in \mc S_k^+$ where  $BCB^T$ has some prescribed zero-nonzero pattern. 

In Remark \ref{rem:interpretation} we have seen an interpretation of $\spr(G)$ as the minimal number of subsets of a given set $V$ that need to be formed if we want to organize required and forbidden interactions as given by the graph $G$.   This interpretation motivates variations of the question considered in this paper. For example, minimizing the number of meetings  corresponds to finding set-join covers $\C$ containing minimal number of set-joins.  Posing restrictions on the size of groups would mean looking for set-join covers $\C$ with restrictions on the cardinalities of the elements in $V(\C)$. In addition to having pairs of elements that are required and pairs that are forbidden to interact, we may have pairs of elements that can (but are not required to) interact. To solve this question we would be looking for partial set-join covers of $G$. 

\section*{Acknowledgments}
Damjana Kokol Bukov\v{s}ek acknowledges financial support from the ARIS (Slovenian Research and Innovation Agency, research core funding No. P1-0222).

\bibliographystyle{amsplain}
\bibliography{biblio}

\end{document}